\documentclass[11 pt]{amsart}

\usepackage{amsfonts}
\usepackage{amssymb}
\usepackage{amsmath}
\usepackage{latexsym}
\usepackage{mathrsfs}
\usepackage{amsthm}
\usepackage{verbatim}

\newtheorem{thrm}{Theorem}[section]
\newtheorem{lem}[thrm]{Lemma}
\newtheorem{cor}[thrm]{Corollary}
\newtheorem{prop}[thrm]{Proposition}

\theoremstyle{definition}
\newtheorem{defn}[thrm]{Definition}

\newtheorem{exmple}[thrm]{Example}
\newtheorem{rmk}[thrm]{Remark}

\begin{document}

\newcommand{\vol}{\mathrm{vol}}

\newcommand{\Supp}{\mathrm{Supp}}

\newcommand{\Sing}{\mathrm{Sing}}

\newcommand{\ord}{\mathrm{ord}}

\newcommand{\Mov}{\mathrm{Mov}}

\newcommand{\mov}{\mathrm{mov}}

\title{The Movable Cone via Intersections}
\author{Brian Lehmann}
\thanks{The author is supported by NSF Award 1004363.}
\address{Department of Mathematics, Rice University \\
Houston, TX \, \, 77005}
\email{blehmann@rice.edu}

\begin{abstract}
We characterize the movable cone of divisors using intersections against curves on birational models.
\end{abstract}

\maketitle

\section{Introduction}

Cones of divisors play an essential role in describing the birational geometry of a smooth complex projective variety $X$.  A key feature of these cones is their interplay with cones of curves via duality statements. The dual of the nef cone and the pseudo-effective cone of divisors were determined by \cite{kleiman66} and \cite{bdpp04} respectively.  We consider the third cone commonly used in birational geometry: the movable cone of divisors.

\begin{defn}
Let $X$ be a smooth projective variety over $\mathbb{C}$.  The movable cone $\overline{\Mov}^{1}(X) \subset N^{1}(X)$ is the closure of the cone generated by classes of effective Cartier divisors $L$ such that the base locus of $|L|$ has codimension at least $2$.  We say a divisor is movable if its numerical class lies in $\overline{\Mov}^{1}(X)$.
\end{defn}

\begin{defn}
Let $X$ be a smooth projective variety over $\mathbb{C}$.  We say that an irreducible curve $C$ on $X$ is movable in codimension $1$, or a $\mov^{1}$-curve, if it deforms to cover a codimension $1$ subset of $X$.
\end{defn}

It is natural to guess that a divisor $L$ is movable if and only if it has non-negative intersection with every $\mov^{1}$-curve.  This is false, as demonstrated by \cite{payne06} Example 1.  Nevertheless, Debarre and Lazarsfeld have asked whether one can formulate a duality statement for movable divisors and $\mov^{1}$-curves.  This has been accomplished for toric varieties in \cite{payne06} and for Mori Dream Spaces in \cite{choi10} by taking other birational models of $X$ into account.  Our main theorem proves an analogous statement for all smooth varieties.

Before stating this theorem, we need to analyze the behavior of the movable cone under birational transformations.  Suppose that $\phi: Y \to X$ is a birational map of smooth projective varieties and that $L$ is a movable divisor on $X$.  It is possible that $\phi^{*}L$ is not movable -- for example, some $\phi$-exceptional centers could be contained in the base locus of $L$.  The following definition from \cite{nakayama04} allows us to quantify the loss in movability.

\begin{defn}
Let $X$ be a smooth projective variety over $\mathbb{C}$ and let $L$ be a pseudo-effective $\mathbb{R}$-divisor on $X$.  Fix an ample divisor $A$ on $X$.  For any prime divisor $\Gamma$ on $X$ we define
\begin{equation*}
\sigma_{\Gamma}(L) = \inf\{\mathrm{mult}_{\Gamma}(L') | L' \geq 0 \textrm{ and } L' \sim_{\mathbb{R}} L + \epsilon A \textrm{ for some } \epsilon > 0 \}
\end{equation*}
where $\sim_{\mathbb{R}}$ denotes $\mathbb{R}$-linear equivalence.  As demonstrated by \cite{nakayama04} III.1.5 Lemma, $\sigma_{\Gamma}$ is independent of the choice of $A$.
\end{defn}

Suppose that $E$ is an exceptional divisor for a birational map $\phi: Y \to X$.  The $\mathbb{R}$-divisor $\sigma_{E}(\phi^{*}L)E$ represents the ``extra contribution'' from $E$ to the non-movability of $\phi^{*}L$.  By subtracting these contributions, we can understand the geometry of the original divisor $L$.

\begin{defn} \label{movtrans}
Let $X$ be a smooth projective variety over $\mathbb{C}$ and let $L$ be a pseudo-effective $\mathbb{R}$-divisor on $X$.  Suppose that $\phi: Y \to X$ is a birational map from a smooth variety $Y$.  The movable transform of $L$ on $Y$ is defined to be
\begin{equation*}
\phi^{-1}_{\mov}(L) := \phi^{*}L - \sum_{E \, \, \phi \mathrm{-exceptional}} \sigma_{E}(\phi^{*}L)E.
\end{equation*}
\end{defn}

Note that the movable transform is not linear and is only defined for pseudo-effective divisors.  We can now state our main theorem.

\begin{thrm} \label{maintheorem}
Let $X$ be a smooth projective variety over $\mathbb{C}$ and let $L$ be a pseudo-effective $\mathbb{R}$-divisor.  $L$ is not movable if and only if there is a $\mov^{1}$-curve $C$ on $X$ and a birational morphism $\phi: Y \to X$ from a smooth variety $Y$ such that
\begin{equation*}
\phi^{-1}_{mov}(L) \cdot \widetilde{C} < 0
\end{equation*}
where $\widetilde{C}$ is the strict transform of a generic deformation of $C$.
\end{thrm}

There does not seem to be an easy way to translate Theorem \ref{maintheorem} into a statement involving only intersections on $X$.  This is a symptom of the fact that the natural operation on movable divisors is the push-forward and not the pull-back.

The proof of Theorem \ref{maintheorem} is accomplished by reinterpreting the orthogonality theorem of \cite{bdpp04} and \cite{bfj09} using the techniques of \cite{lehmann10}.


\begin{exmple}
For surfaces Theorem \ref{maintheorem} reduces to the usual duality of the nef and pseudo-effective cones.
\end{exmple}

\begin{exmple}
Suppose that $X$ is a smooth Mori dream space and $L$ is an $\mathbb{R}$-divisor on $X$.  By running the $L$-MMP as in \cite{hk00}, we obtain a small modification $\phi: X \dashrightarrow X'$, a morphism $f: X \to Z$, and an ample $\mathbb{R}$-divisor $A$ on $Z$ such that
\begin{equation*}
\phi_{*}^{-1}L \equiv f^{*}A
\end{equation*}
where $\phi_{*}^{-1}$ denotes the strict transform.

Let $W$ be a smooth variety admitting birational maps $\psi: W \to X$ and $\psi': W \to X'$.  Using \cite{nakayama04} III.5.5 Proposition, one easily verifies that
\begin{equation*}
\psi^{-1}_{mov}(L) \equiv \psi'^{*}(\phi_{*}^{-1}L).
\end{equation*}
Thus Theorem \ref{maintheorem} implies the statements of \cite{payne06} and \cite{choi10}: for a smooth toric variety or Mori Dream Space $X$, a divisor class is movable iff its strict transform class on every $\mathbb{Q}$-factorial small modification $X'$ has non-negative intersection with every $\mov^{1}$-curve on $X'$.
\end{exmple}

\begin{exmple}
Suppose that $X$ is a smooth projective variety with $K_{X}$ numerically trivial.  \cite{choi10} explains how to apply techniques of the minimal model program to analyze $\overline{\Mov}^{1}(X)$.  Just as before, a divisor class is movable if and only if its strict transform class on every $\mathbb{Q}$-factorial small modification has non-negative intersection with every $\mov^{1}$-curve.  When $X$ is hyperk\"ahler, \cite{huybrechts03} and \cite{boucksom04} show that in fact it suffices to consider small modifications that are also smooth hyperk\"ahler varieties.

More generally, \cite{choi10} shows that small modifications can detect certain regions of $\overline{\Mov}^{1}(X)$ by using the minimal model program.
\end{exmple}

We will also prove a slightly stronger version of Theorem \ref{maintheorem} that involves the non-nef locus $\mathbf{B}_{-}(L)$ of $L$ (which will be defined in Definition \ref{nonnefdef}).  Although the non-nef locus represents the ``obstruction'' to the nefness of $L$, it is not true that $\mathbf{B}_{-}(L)$ is covered by curves $C$ with $L \cdot C < 0$.  However, Proposition \ref{mainprop} formulates a birational version of this negativity using the movable transform.

Finally, we will use Proposition \ref{mainprop} to understand $k$-movability for $k>1$.  Define the $k$-movable cone of $X$ to be the closure of the cone in $N^{1}(X)$ generated by effective Cartier divisors whose base locus has codimension at least $k-1$.  We say that a divisor is $k$-movable if its numerical class lies in the $k$-movable cone.  Note that the $1$-movable cone is just $\overline{\Mov}^{1}(X)$.

Debarre and Lazarsfeld have asked whether there is a duality between the $k$-movable cone of divisors and the closure of the cone of irreducible curves that deform to cover a codimension $k$ subset (for $0 < k < \dim X$).  Corollary \ref{maincor} constructs a birational version of this duality.  Again, this generalizes results for toric varieties in \cite{payne06} and for Mori dream spaces in \cite{choi11}.

\section{Background}

Throughout $X$ will denote a smooth projective variety over $\mathbb{C}$.  We use the notations $\sim, \sim_{\mathbb{Q}}, \sim_{\mathbb{R}}, \equiv$ to denote respectively linear equivalence, $\mathbb{Q}$-linear equivalence, $\mathbb{R}$-linear equivalence, and numerical equivalence of $\mathbb{R}$-divisors.  The volume of an $\mathbb{R}$-divisor $L$ is
\begin{equation*}
\vol_{X}(L) = \limsup_{m \to \infty} \frac{h^{0}(X,\lfloor mL \rfloor)}{m^{\dim X}}.
\end{equation*}

\subsection{Divisorial Zariski decomposition}

Let $L$ be a pseudo-effective $\mathbb{R}$-divisor on a smooth projective variety $X$.  Recall that for a prime divisor $\Gamma$ on $X$ we have defined
\begin{equation*}
\sigma_{\Gamma}(L) = \inf\{\mathrm{mult}_{\Gamma}(L') | L' \geq 0 \textrm{ and } L' \sim_{\mathbb{R}} L + \epsilon A \textrm{ for some } \epsilon > 0 \}.
\end{equation*}
where $A$ is any fixed ample divisor.  \cite{nakayama04} III.1.11 Corollary shows that there are only finitely many prime divisors $\Gamma$ on $X$ with $\sigma_{\Gamma}(L) > 0$, allowing us to make the following definition.

\begin{defn}[\cite{nakayama04} III.1.16 Definition] \label{zardecom}
Let $L$ be a pseudo-effective $\mathbb{R}$-divisor on $X$.  Define
\begin{equation*}
N_{\sigma}(L) = \sum \sigma_{E}(L) E \qquad \qquad P_{\sigma}(L) = L - N_{\sigma}(L)
\end{equation*}
The decomposition $L = N_{\sigma}(L) + P_{\sigma}(L)$ is called the \emph{divisorial Zariski
decomposition} of $L$.
\end{defn}

Note that for a birational morphism $\phi: Y \to X$ we have $\phi_{mov}^{-1}(L) = P_{\sigma}(\phi^{*}L) + \phi_{*}^{-1}N_{\sigma}(L)$ where $\phi_{*}^{-1}$ denotes the strict transform.  The divisorial Zariski decomposition is closely related to the non-nef locus of $L$.

\begin{defn} \label{nonnefdef}
Let $X$ be a smooth projective variety and let $L$ be a pseudo-effective $\mathbb{R}$-divisor on $X$.  We define the $\mathbb{R}$-stable base locus of $L$ to be the subset of $X$ given by
\begin{equation*}
\mathbf{B}_{\mathbb{R}}(L) = \bigcup \{ \Supp(L') | L' \geq 0 \textrm{ and }L' \sim_{\mathbb{R}} L \}.
\end{equation*}
The non-nef locus of $L$ is then defined to be
\begin{equation*}
\mathbf{B}_{-}(L) = \bigcup_{A \textrm{ ample }\mathbb{R}\textrm{-divisor}}  \mathbf{B}_{\mathbb{R}}(L+A).
\end{equation*}
\end{defn}

The following proposition records the basic properties of the divisorial Zariski decomposition.

\begin{prop}[\cite{nakayama04} III.1.14 Proposition, III.2.5 Lemma, V.1.3 Theorem]
Let $X$ be a smooth projective variety and let $L$ be a pseudo-effective $\mathbb{R}$-divisor.
\begin{enumerate}
\item $P_{\sigma}(L)$ is a movable $\mathbb{R}$-divisor.  In particular for any prime divisor $E$ the restriction $P_{\sigma}(L)|_{E}$ is pseudo-effective.
\item If $\phi: Y \to X$ is a birational morphism of smooth varieties and $\Gamma$ is a prime divisor on $Y$ that is not $\phi$-exceptional, then $\sigma_{\Gamma}(\phi^{*}L) = \sigma_{\phi(\Gamma)}(L)$.
\item The union of the codimension $1$ components of $\mathbf{B}_{-}(L)$ coincides with $\Supp(N_{\sigma}(L))$.
\end{enumerate}
\end{prop}

\subsection{Numerical dimension and orthogonality}

Given a pseudo-effective divisor $L$, the numerical dimension $\nu(L)$ of \cite{nakayama04} and \cite{bdpp04} is a numerical measure of the ``positivity'' of $L$.  There is also a restricted variant $\nu_{X|V}(L)$ introduced in \cite{bfj09}; since the definition is somewhat involved, we will only refer to a special subcase using an alternate characterization from \cite{lehmann10}.

\begin{defn} \label{numdimdef}
Let $L$ be a pseudo-effective divisor on $X$.  Fix a prime divisor $E$ on $X$ and choose $L' \equiv L$ whose support does not contain $E$.  We say $\nu_{X|E}(L) = 0$ if
\begin{equation*}
\liminf_{\phi} \vol_{\widetilde{E}}(P_{\sigma}(\phi^{*}L')|_{\widetilde{E}}) = 0
\end{equation*}
where $\phi: \widetilde{X} \to X$ varies over all birational maps and $\widetilde{E}$ denotes the strict transform of $E$.
\end{defn}

The connection with geometry is given by the following version of the orthogonality theorem of \cite{bdpp04} and \cite{bfj09}.

\begin{thrm}[\cite{bfj09}, Theorem 4.15]
Let $L$ be a pseudo-effective divisor.  If a prime divisor $E \subset X$ is contained in $\Supp(N_{\sigma}(L))$ then $\nu_{X|E}(L) = 0$.
\end{thrm}

\begin{proof}
Fix an ample divisor $A$.  Choose an $\epsilon > 0$ sufficiently small so that $\Supp(N_{\sigma}(L)) = \Supp(N_{\sigma}(L + \epsilon A))$ and apply \cite{bfj09} Theorem 4.15.  The comparison between the numerical dimension of \cite{bfj09} and Definition \ref{numdimdef} is given by \cite{lehmann10} Theorem 7.1.
\end{proof}

\section{Proof}

\begin{proof}[Proof of Theorem \ref{maintheorem}:]

Suppose that $L$ is not movable.  Denote by $E$ a fixed divisorial component of $N_{\sigma}(L)$.

Fix a sufficiently general ample divisor $A$ on $X$ and choose $\epsilon$ small enough so that $E$ is a component of $N_{\sigma}(L+\epsilon A)$.  Applying the orthogonality theorem of \cite{bdpp04}, we see there is a birational map $\phi: Y \to X$ so that
\begin{enumerate}
\item $\widetilde{E}$ is smooth.
\item $\vol_{\widetilde{E}}(P_{\sigma}(\phi^{*}(L+\epsilon A))|_{\widetilde{E}}) < \vol_{E}(A|_{E}) = \vol_{\widetilde{E}}(\phi^{*}A|_{\widetilde{E}})$.
\item The strict transform of every component of $N_{\sigma}(L)$ is disjoint.
\end{enumerate}
There is a unique expression
\begin{equation*}
P_{\sigma}(\phi^{*}(L+\epsilon A)) = P_{\sigma}(\phi^{*}L) + \phi^{*}A + \alpha(\epsilon)\widetilde{E} + F
\end{equation*}
where $\widetilde{E}$ is the strict transform of $E$,  $F$ is an effective divisor with $F \leq N_{\sigma}(\phi^{*}L)$ and the support of $F$ does not contain $E$, and $\alpha(\epsilon)$ is positive and goes to $0$ as $\epsilon$ goes to $0$.  By shrinking $\epsilon$ we may ensure that $\alpha(\epsilon) < \sigma_{E}(L)$.

Condition (2) above, along with Lemma \ref{vollemma}, show that  the restriction $(P_{\sigma}(\phi^{*}L) + \alpha(\epsilon)\widetilde{E})|_{\widetilde{E}}$ is not pseudo-effective for any $\epsilon > 0$.  Since $\alpha(\epsilon) < \sigma_{E}(L)$, we also have that $(P_{\sigma}(\phi^{*}L) + \sigma_{E}(L)\widetilde{E})|_{\widetilde{E}}$ is not pseudo-effective.  As the strict transform of components of $N_{\sigma}(L)$ are disjoint, the restriction of $P_{\sigma}(\phi^{*}L) + \phi_{*}^{-1} N_{\sigma}(L)$ to $\widetilde{E}$ is still not pseudo-effective.

By \cite[0.2 Theorem]{bdpp04} there is a curve $\widetilde{C}$ whose deformations cover $\widetilde{E}$ such that $$(P_{\sigma}(\phi^{*}L) + \phi_{*}^{-1} N_{\sigma}(L)) \cdot \widetilde{C} < 0.$$  Since $\widetilde{E}$ is not $\phi$-exceptional, $C = \phi(\widetilde{C})$ is a $\mov^{1}$-curve.

Conversely, if $L$ is movable, then $\phi^{-1}_{mov}(L) = P_{\sigma}(\phi^{*}L)$ is also movable for every $\phi$.  Thus every movable transform has non-negative intersection with the strict transform of every $\mov^{1}$-curve general in its family.
\end{proof}

\begin{lem} \label{vollemma}
Let $X$ be a smooth projective variety and let $L$ and $L'$ be pseudo-effective divisors on $X$.  Then $\vol_{X}(L + L') \geq \vol_{X}(L)$.
\end{lem}

\begin{proof}
We may assume $L$ is big since otherwise the inequality is automatic.  Then for any sufficiently small $\epsilon > 0$ we have
$$\vol_{X}(L+L') = \vol_{X}((1-\epsilon) L + (\epsilon L + L')) \geq (1-\epsilon)^{\dim X}\vol_{X}(L)$$
since $\epsilon L + L'$ is big.
\end{proof}

We now give an alternate formulation of Theorem \ref{maintheorem}.

\begin{prop} \label{mainprop}
Let $X$ be a smooth projective variety and let $L$ be a pseudo-effective $\mathbb{R}$-divisor.  Suppose that $V$ is an irreducible subvariety of $X$ contained in $\mathbf{B}_{-}(L)$ and let $\psi: X' \to X$ be a smooth birational model resolving the ideal sheaf of $V$.  Then there is a birational morphism $\phi: Y \to X'$ from a smooth variety $Y$ and an irreducible curve $\widetilde{C}$ on $Y$ such that
\begin{equation*}
\phi^{-1}_{mov}(\psi^{*}L) \cdot \widetilde{C} < 0
\end{equation*}
and $\psi \circ \phi(\widetilde{C})$ deforms to cover $V$.
\end{prop}

\begin{proof}
Let $E$ be the $\psi$-exceptional divisor dominating $V$.  Since we have $E \subset \Supp(N_{\sigma}(\psi^{*}L))$, we may argue as in the proof of Theorem \ref{maintheorem} for $\psi^{*}L$ and $E$ to find a birational map $\phi$ such that $\phi^{-1}_{mov}(\psi^{*}L)|_{\widetilde{E}}$ is not pseudo-effective.

\cite[2.4 Theorem]{bdpp04} shows that there is some curve $\widetilde{C}$ on $\widetilde{E}$ with $\phi^{-1}_{mov}(\psi^{*}L) \cdot \widetilde{C} < 0$ such that $\widetilde{C}$ deforms to cover $\widetilde{E}$ and is not contracted by any morphism from $\widetilde{E}$ to a variety of positive dimension.  Choosing $\widetilde{C}$ on $\widetilde{E}$ to satisfy this stronger property, we obtain the statement of Proposition \ref{mainprop}.
\end{proof}

Proposition \ref{mainprop} shows that the non-nef locus is covered by $L$-negative curves in a birational sense.  Alternatively, one can rephrase this result using $k$-movability.

\begin{cor} \label{maincor}
Let $X$ be a smooth projective variety and let $L$ be a pseudo-effective $\mathbb{R}$-divisor.  Then $L$ is not $k$-movable if and only if there is a birational morphism $\psi: X' \to X$ from a smooth variety $X'$, a birational morphism $\phi: Y \to X'$ from a smooth variety $Y$, and an irreducible curve $\widetilde{C}$ on $Y$ such that
\begin{equation*}
\phi^{-1}_{mov}(\psi^{*}L) \cdot \widetilde{C} < 0
\end{equation*}
and $\psi \circ \phi(\widetilde{C})$ deforms to cover a $k$-dimensional subset of $V$.
\end{cor}

\begin{proof}
To say that $L$ is not $k$-movable is equivalent to saying that $\mathbf{B}_{-}(L)$ has a component of dimension at least $k$.  Apply Proposition \ref{mainprop} to obtain the forward implication.  The converse is immediate.
\end{proof}

\begin{rmk}
It is unclear whether Corollary \ref{maincor} is the best formulation possible for the duality of $k$-movable divisors.  For Mori Dream Spaces varieties and for $2 < k < \dim X$, \cite{payne06} Theorem 1 and \cite{choi11} Corollary 3 prove a slightly stronger statement.  The essential difference is that one does not need to blow-up along top-dimensional components of $\mathbf{B}_{-}(L)$.  More precisely, if $L$ is not $k$-movable, one may find a $\mathbb{Q}$-factorial small modification $f: X \dashrightarrow X'$ that is regular at the generic point of a component $V \subset \mathbf{B}_{-}(L)$ of codimension at most $k$ and a family of curves covering the strict transform of $V$ with $f_{*}L \cdot C < 0$.  In contrast, Corollary \ref{maincor} may produce a birational map that is not regular at any point of $V$.
\end{rmk}

\nocite{*}
\bibliographystyle{amsalpha}
\bibliography{movcone}

\end{document}